\documentclass{amsart}
\usepackage{amstext,amssymb,amsthm,amsopn,newlfont,graphpap,graphics,graphicx,mathrsfs,enumitem,tikz-cd}
\usepackage[parfill]{parskip}
\usepackage[noadjust]{cite}
\usepackage{epigraph}
\usepackage[colorlinks=true,
            linkcolor=red,
            urlcolor=blue,
            citecolor=magenta]{hyperref}
\usepackage{color}
\usepackage{mathrsfs}
\allowdisplaybreaks
\theoremstyle{plain}
\newtheorem{thm}{Theorem}[section]
\newtheorem*{thm*}{Theorem}
\newtheorem{prop}{Proposition}[section]
\newtheorem*{prop*}{Proposition}

\newtheorem*{cor*}{Corollary}
\newtheorem{lem}{Lemma}[section]
\newtheorem*{lem*}{Lemma}
\theoremstyle{definition}
\newtheorem{defn}{Definition}[section]
\newtheorem*{defn*}{Definition}

\newtheorem*{exmp*}{Example}
\newtheorem{exmps}{Examples}[section]
\newtheorem*{exmps*}{Examples}

\newtheorem{rem}{Remark}[section]
\newtheorem*{rem*}{Remark}
\newtheorem{rems}{Remarks}[section]
\newtheorem*{rems*}{Remarks}

\newtheorem*{note*}{Note}
\newcommand{\N}{{\mathbb N}}
\newcommand{\Z}{{\mathbb Z}}
\newcommand{\R}{{\mathbb R}}
\newcommand{\C}{{\mathbb C}}
\newcommand{\F}{{\mathbb F}}
\newcommand{\emps}{\emptyset}

\renewcommand{\bar}{\overline}

\numberwithin{equation}{section}
\DeclareMathOperator{\Rep}{Re\ignorespaces}

\DeclareMathOperator{\orb}{orb}

\begin{document}
\title[The non-hypercyclicity of normal operators and their exponentials]
{On the non-hypercyclicity of normal operators, collections of their exponentials,\\
and symmetric operators}
\author[Marat V. Markin]{Marat V. Markin}
\address{
Department of Mathematics\newline
California State University, Fresno\newline
5245 N. Backer Avenue, M/S PB 108\newline
Fresno, CA 93740-8001
}
\email{mmarkin@csufresno.edu}
\author{Edward S. Sichel}
\email[Corresponding author]{edsichel@mail.fresnostate.edu}
\begin{abstract}
We give a simple, straightforward proof of the non-hypercyclicity of an arbitrary  \textit{normal operator} $A$ (bounded or unbounded) in a complex Hilbert space as well as of the collection $\left\{e^{tA}\right\}_{t\ge 0}$ of its exponentials, which, under a certain condition on the spectrum of $A$, is the $C_0$-semigroup generated by it. We also establish non-hypercyclicity for \textit{symmetric operators}.
\end{abstract}
\maketitle

\section[Introduction]{Introduction}

In \cite{Markin2020(1)}, a straightforward proof is furnished of the non-hypercyclicity of a \textit{scalar type spectral operator} $A$ (bounded or unbounded) in a complex Banach space as well as of the collection $\left\{e^{tA}\right\}_{t\ge 0}$ of its exponentials (see, e.g., \cite{Dun-SchIII}). The important special case of a \textit{normal operator} $A$ in a complex Hilbert space (see, e.g., \cite{Dun-SchII,Plesner}) follows immediately.

Here, we furnish a shorter, simpler, and more transparent direct proof for a normal operator (bounded or unbounded) in a complex Hilbert space, which does not require utilizing the dual space, generalizing the known result for bounded normal operators \cite[Corollary $5.31$]{Grosse-Erdmann-Manguillot}. We further establish non-hypercyclicity for \textit{symmetric operators} (see, e.g., \cite{Akh-Glaz}).


\section[Preliminaries]{Preliminaries}

Here, we briefly outline certain preliminaries essential for the subsequent discourse.

\begin{defn}[Hypercyclicity]
For a (bounded or unbounded) linear operator $A$ in a (real or complex) Banach space $X$, a nonzero vector 
\begin{equation*}
f\in C^\infty(A):=\bigcap_{n=0}^{\infty}D(A^n)
\end{equation*}
($D(\cdot)$ is the \textit{domain} of an operator, $A^0:=I$, $I$ is the \textit{identity operator} on $X$) is called \textit{hypercyclic} if its \textit{orbit} under $A$
\[
\orb(f,A):=\left\{A^nf\right\}_{n\in\Z_+}
\]
($\Z_+:=\left\{0,1,2,\dots\right\}$ is the set of nonnegative integers) is dense in $X$, i.e.,
\[
\bar{\orb(f,A)}=X
\]
($\bar{\cdot}$ is the closure of a set).

Linear operators possessing hypercyclic vectors are said to be \textit{hypercyclic}.

More generally, a collection $\left\{T(t)\right\}_{t\in J}$ ($J$ is a nonempty indexing set) of linear operators in $X$ is called \textit{hypercyclic} if there exists a 
\textit{hypercyclic vector}, i.e., a vector
$\displaystyle f\in \bigcap_{t\in J}D(T(t))$ whose \textit{orbit} under $\left\{T(t)\right\}_{t\in J}$
\[
\orb\left(f,\left\{T(t)\right\}_{t\in J}\right):=\left\{T(t)f\right\}_{t\in J}
\]
is dense in $X$,
i.e.,
\[
\bar{\orb\left(f,\left\{T(t)\right\}_{t\in J}\right)}=X.
\]

Cf. \cite{Grosse-Erdmann-Manguillot,Guirao-Montesinos-Zizler,Rolewicz1969,B-Ch-S2001,deL-E-G-E2003,arXiv:2203.02032}.
\end{defn}

\begin{rems}\label{remshlo}\
\begin{itemize}
\item In the prior definition of hypercyclicity, for a linear operator, the underlying space must necessarily be
\textit{infinite-dimensional} and \textit{separable} (see, e.g., \cite{Grosse-Erdmann-Manguillot}), whereas, for a collection of linear operators, the space need not be separable.
\item For a hypercyclic linear operator $A$, the set $HC(A)$ of its hypercyclic vectors is necessarily dense in $X$ and therefore so is the subspace $C^\infty(A)\supseteq HC(A)$. In particular, this implies that, for any $n\in \N$,
the operator $A^n$ must be \textit{densely defined} (i.e., $\bar{D\left(A^n\right)}=X$).
\item Bounded normal operators on a complex Hilbert space are known to be non-hypercyclic \cite[Corollary $5.31$]{Grosse-Erdmann-Manguillot}. 
\end{itemize} 
\end{rems} 

Henceforth, unless specified otherwise, $A$ is a {\it normal operator} in a complex Hilbert space $(X,(\cdot,\cdot),\|\cdot\|)$ with strongly $\sigma$-additive \textit{spectral measure} (the \textit{resolution of the identity}) $E_A(\cdot)$ assigning to Borel sets of the complex plane $\C$ orthogonal projection operators on $X$ and having the operator's \textit{spectrum} $\sigma(A)$ as its {\it support} \cite{Dun-SchII,Plesner}.

Associated with a normal operator $A$ is the {\it Borel operational calculus} assigning to any Borel measurable function $F:\sigma(A)\to \C$ a normal operator
\begin{equation*}
F(A):=\int\limits_{\sigma(A)} F(\lambda)\,dE_A(\lambda)
\end{equation*}
with
\[
D(F(A))=\left\{ f\in X \,\middle|\, \int\limits_{\sigma(A)} |F(\lambda)|^2\,d(E_A(\lambda)f,f)<\infty \right\},
\]
where $(E_A(\cdot)f,f)$ is a Borel measure, in which case
\begin{equation}\label{n2}
\|F(A)f\|^2=\int\limits_{\sigma(A)} |F(\lambda)|^2\,d(E_A(\lambda)f,f)
\end{equation}
\cite{Dun-SchII,Plesner}.

In particular,
\begin{equation*}
A^n=\int\limits_{\sigma(A)} \lambda^n\,dE_A(\lambda),\ n\in\Z_+,\quad\text{and}\quad
e^{tA}:=\int\limits_{\sigma(A)} e^{t\lambda}\,dE_A(\lambda),\ t\in \R.
\end{equation*}

Provided
\[
\sigma(A)\subseteq \left\{\lambda\in\C\,\middle|\, \Rep\lambda\le \omega\right\}
\] 
with some $\omega\in \R$, the collection 
of exponentials $\left\{e^{tA}\right\}_{t\ge 0}$
is the $C_0$-\textit{semigroup} generated by $A$ \cite{Engel-Nagel,Plesner}.

\begin{rems}\label{remsws}\
\begin{itemize}
\item By \cite[Theorem $3.1$]{Markin1999},
the orbits
\begin{equation}\label{expf1}
y(t)=e^{tA}f,\ t\ge 0,f \in \bigcap_{t\ge 0}D(e^{tA}),
\end{equation}
describe all \textit{weak/mild solutions} of the abstract evolution equation
\begin{equation}\label{+}
y'(t)=Ay(t),\ t\ge 0,
\end{equation}
(see \cite{Ball}, cf. also {\cite[Ch. II, Definition 6.3]{Engel-Nagel}}).
\item The subspaces
\[
C^\infty(A)\quad \text{and}\quad \bigcap_{t\ge 0}D(e^{tA}) 
\]
of all possible initial values for the corresponding orbits are \textit{dense} in $X$ since they contain the subspace
\begin{equation*}
\bigcup_{\alpha>0}E_A(\Delta_\alpha)X,\ \text{where}\ \Delta_\alpha:=\left\{\lambda\in\C\,\middle|\,|\lambda|\le \alpha \right\},\ \alpha>0,
\end{equation*}
which is dense in $X$ and coincides with the class ${\mathscr E}^{\{0\}}(A)$ of the \textit{entire vectors of $A$ of exponential type} (see, e.g., \cite{Gor-Knyaz,Radyno1983(1)}, cf. also \cite{Markin2015}).
\end{itemize} 
\end{rems} 

\section{Normal Operators and Their Exponentials}

Here, we prove the special case of \cite[Theorem $3.1$]{Markin2020(1)}
for a normal operator and its exponentials directly, in particular generalizing \cite[Corollary $5.31$]{Grosse-Erdmann-Manguillot}.

\begin{thm}\label{Thm1}
An arbitrary normal, in particular self-adjoint, operator $A$ in a nonzero complex Hilbert space $(X,(\cdot,\cdot),\|\cdot\|)$ with spectral measure $E_A(\cdot)$ is not hypercyclic and neither is the collection $\left\{e^{tA}\right\}_{t\ge 0}$ of its exponentials, which, provided the spectrum of $A$ is located in a left half-plane
\[
\left\{\lambda\in \C\,\middle|\,\Rep\lambda\le \omega \right\}
\]
with some $\omega\in\R$, is the $C_0$-semigroup generated by $A$. 
\end{thm}

\begin{proof}
Let $f\in C^\infty(A)\setminus \{0\}$ be arbitrary.

There are two possibilities: either
\[
E_A\left(\left\{\lambda\in \sigma(A)\,\middle|\, |\lambda|>1\right\}\right)f\neq 0
\]
or
\[
E_A\left(\left\{\lambda\in \sigma(A)\,\middle|\, |\lambda|>1\right\}\right)f=0.
\]

In the first case, for any $n\in \Z_+$,
\begin{multline*}
{\left\|A^nf\right\|}^2
\hfill
\text{by \eqref{n2}};
\\
\shoveleft{
=\int\limits_{\sigma(A)}{|\lambda|}^{2n}\,d(E_A(\lambda)f,f)
\ge 
\int\limits_{\{\lambda\in\sigma(A)\,|\,|\lambda|>1\}}{|\lambda|}^{2n}\,d(E_A(\lambda)f,f)
}\\
\shoveleft{
\ge 
\int\limits_{\{\lambda\in\sigma(A)\,|\,|\lambda|>1\}}1\,d(E_A(\lambda)f,f)
= \left(E_A(\{\lambda\in\sigma(A)\,|\,|\lambda|>1\})f,f\right)
}\\
\ \ \
=\left\|E_A(\{\lambda\in\sigma(A)\,|\,|\lambda|>1\})f\right\|^2>0,
\hfill
\end{multline*}
which implies that the orbit $\orb(f,A)$ of $f$ under $A$ cannot approximate the zero vector, and hence, is not dense in $X$.

In the second case, since
\[
f=E_A\left(\left\{\lambda\in \sigma(A)\,\middle|\, |\lambda|>1\right\}\right)f
+
E_A\left(\left\{\lambda\in \sigma(A)\,\middle|\, |\lambda|\le 1\right\}\right)f,
\]
we infer that
\[
f=E_A\left(\left\{\lambda\in \sigma(A)\,\middle|\, |\lambda|\le 1\right\}\right)f\neq 0
\]
and hence, for any $n\in \Z_+$,
\begin{multline*}
\left\|A^nf\right\|^2
=\left\|A^nE_A\left(\left\{\lambda\in \sigma(A)\,\middle|\, |\lambda|\le 1\right\}\right)f\right\|^2
\\
\hfill
\text{by \eqref{n2} and the properties of the \textit{operational calculus}};
\\
\shoveleft{
=
\int\limits_{\{\lambda\in\sigma(A)\,|\,|\lambda|\le 1\}}{|\lambda|}^{2n}\,d(E_A(\lambda)f,f)
\le
\int\limits_{\{\lambda\in\sigma(A)\,|\,|\lambda|\le 1\}}1\,d(E_A(\lambda)f,f)
}\\
\ \ \
= \left(E_A(\{\lambda\in\sigma(A)\,|\,|\lambda|\le 1\})f,f\right)
=\left\|E_A(\{\lambda\in\sigma(A)\,|\,|\lambda|\le 1\})f\right\|^2=\|f\|^2,
\hfill
\end{multline*}
which also implies that the orbit $\orb(f,A)$ of $f$ under $A$, being bounded, is not dense in $X$ and completes the proof for the operator case.

\smallskip
Now, let us consider the case of the exponential collection $\left\{e^{tA}\right\}_{t\ge 0}$ assuming that $\displaystyle f \in \bigcap_{t\ge 0}D(e^{tA})\setminus \{0\}$ is arbitrary.

There are two possibilities: either
\[
E_A\left(\left\{\lambda\in \sigma(A)\,\middle|\, \Rep\lambda>0\right\}\right)f\neq 0
\]
or
\[
E_A\left(\left\{\lambda\in \sigma(A)\,\middle|\, \Rep\lambda>0\right\}\right)f=0.
\]

In the first case, for any $t\ge 0$,
\begin{multline*}
\|e^{tA}f\|^2
\hfill
\text{by \eqref{n2}};
\\
\shoveleft{
=\int\limits_{\sigma(A)}{\left|e^{\lambda}\right|}^2\,d(E_A(\lambda)f,f)
=\int\limits_{\sigma(A)}e^{2t\Rep\lambda}\,d(E_A(\lambda)f,f)
}\\
\shoveleft{
\ge 
\int\limits_{\{\lambda\in\sigma(A)\,|\,\Rep\lambda>0\}}e^{2t\Rep\lambda}\,d(E_A(\lambda)f,f)
\ge 
\int\limits_{\{\lambda\in\sigma(A)\,|\,\Rep\lambda>0\}}1\,d(E_A(\lambda)f,f)
}\\
\ \ \
= \left(E_A(\{\lambda\in\sigma(A)\,|\,\Rep\lambda>0\})f,f\right)
=\left\|E_A(\{\lambda\in\sigma(A)\,|\,\Rep\lambda>0\})f\right\|^2>0,
\hfill
\end{multline*}
which implies that the orbit $\orb\left(f,\left\{e^{tA}\right\}_{t\ge 0}\right)$ of $f$ under $\left\{e^{tA}\right\}_{t\ge 0}$ cannot approximate the zero vector, and hence, is not dense in $X$.

In the second case, since
\[
f=E_A\left(\left\{\lambda\in \sigma(A)\,\middle|\, \Rep\lambda>0\right\}\right)f
+
E_A\left(\left\{\lambda\in \sigma(A)\,\middle|\, \Rep\lambda\le 0\right\}\right)f,
\]
we infer that
\[
f=E_A\left(\left\{\lambda\in \sigma(A)\,\middle|\, \Rep\lambda\le 0\right\}\right)f\neq 0
\]
and hence, for any $t\ge 0$,
\begin{multline*}
{\left\|e^{tA}f\right\|}^2
=\left\|e^{tA}E_A\left(\left\{\lambda\in \sigma(A)\,\middle|\, \Rep\lambda\le 0\right\}\right)f\right\|^2
\\
\hfill
\text{by \eqref{n2} and the properties of the \textit{operational calculus}};
\\
\shoveleft{
=
\int\limits_{\{\lambda\in\sigma(A)\,|\,\Rep\lambda\le 0\}}{\left|e^{\lambda}\right|}^2\,d(E_A(\lambda)f,f)
=
\int\limits_{\{\lambda\in\sigma(A)\,|\,\Rep\lambda\le 0\}}e^{2t\Rep\lambda}\,d(E_A(\lambda)f,f)
}\\
\shoveleft{
\le \int\limits_{\{\lambda\in\sigma(A)\,|\,\Rep\lambda\le 0\}}1\,d(E_A(\lambda)f,f)
= \left(E_A(\{\lambda\in\sigma(A)\,|\,\Rep\lambda\le 0\})f,f\right)
}\\
\ \ \
=\left\|E_A(\{\lambda\in\sigma(A)\,|\,\Rep\lambda\le 0\})f\right\|^2=\|f\|^2,
\hfill
\end{multline*}
which also implies that the orbit $\orb\left(f,\left\{e^{tA}\right\}_{t\ge 0}\right)$ of $f$ under $\left\{e^{tA}\right\}_{t\ge 0}$, being bounded, is not dense in $X$ and completes the proof of the exponential case and the entire statement.
\end{proof}

\begin{rem}
A normal operator $A$ in a complex Hilbert space whose spectrum is located on the imaginary axis ($\sigma(A)\subseteq i\R$, where $i$ is the \textit{imaginary unit}), i.e., an \textit{anti-self-adjoint operator} ($A^*=-A$, where $A^*$ is the \textit{adjoint} of $A$), generates a strongly continuous operator group $\left\{e^{tA}\right\}_{t\in \R}$ consisting of \textit{unitary operators} \cite{Stone1932}, which, in particular, implies that
\[
\left\|e^{tA}\right\|=1,\ t\in \R.
\]

Thus, such an operator and the generated operator group are non-hypercyclic 
(cf. \cite[Corollary $3.1$]{Markin2020(1)}).
\end{rem}

\section{Symmetric Operators}

The following generalizes in part \cite[Lemma $2.53$ (a)]{Grosse-Erdmann-Manguillot} to the case of a densely defined unbounded linear operator in a Hilbert space (cf. \cite[Proposition $4.1$ (2)--(4)]{arXiv:2106.14872} for a Banach space setting).

\begin{lem}\label{Lem1}
Let $A$ be a hypercyclic linear operator in a nonzero Hilbert space $(X,(\cdot,\cdot),\|\cdot\|)$ over the scalar field $\F$ of real or complex numbers  (i.e., $\F=\R$ or $\F=\C$). Then
\begin{enumerate}
\item the adjoint operator $A^*$ has no eigenvalues, or equivalently, for any $\lambda \in \F$, the range of the operator $A-\lambda I$ is dense in $X$, i.e.,
\[
\bar{R(A-\lambda I)}=X
\]
($R(\cdot)$ is the range of an operator);
\item provided the space $X$ is complex (i.e., $\F=\C$) and the operator $A$ is closed,
the residual spectrum of $A$ is empty, i.e.,
\[
\sigma_r(A)=\emps.
\]
\end{enumerate}
\end{lem}

\begin{proof}\
\begin{enumerate}
\item Let $f\in X$ be a hypercyclic vector for $A$. 

We proceed \textit{by contradiction}, assuming that
the adjoint operator $A^*$, which exists since $A$ is \textit{densely defined} (see Remarks \ref{remshlo}), has an eigenvalue $\lambda\in \F$, and hence,
\[
\exists\, x\in D(A^*)\setminus \{0\}:\ A^*g=\lambda g,
\]
which, in particular, implies that $g\in C^\infty(A^*):=\bigcap_{n=0}^{\infty}D\left({(A^*)}^n\right)$ and
\[
\forall\, n\in \N:\ {(A^*)}^ng=\lambda^n g.
\]

In view of the above, we have inductively:
\begin{equation}\label{dnd}
\forall\, n\in \N:\  (A^nf,g)=(A^{n-1}f,A^*g)=(f,{(A^*)}^ng)=(f,\lambda^n g)
=\bar{\lambda}^n(f,g),
\end{equation}
the conjugation being superfluous when the space is real.

Since $g\neq 0$, by the \textit{Riesz representation theorem} (see, e.g., \cite{Markin2027EFA,Markin2026EOT}), the hypercyclicity of $f$ implies that the set
\[
\left\{(A^nf,g)\right\}_{n\in \Z_+}
\]
is \textit{dense} in $\F$, being the image of the dense in the space $X$ orbit $\orb(f,A):=\left\{A^nf\right\}_{n\in\Z_+}$ of the hypercyclic vector $f$ under the nonzero continuous linear functional
\[
X\ni h\mapsto (h,g)\in \F.
\]

By \eqref{dnd}, the denseness of the set $\left\{(A^nf,g)\right\}_{n\in \Z_+}$ in $\F$, in particular, implies that 
\[
(f,g)\ne 0,
\]
i.e., the vectors $f$ and $g$ are \textit{not} orthogonal and contradicts the fact that the same set
\[
\left\{{\bar{\lambda}}^n(f,g)\right\}_{n\in \Z_+}
\]
is \textit{not} dense in $\F$. 

Indeed, for $\lambda\in\F$ with $|\lambda|\ge 1$,
\[
\big|{\bar{\lambda}}^n(f,g)\big|={\left|\bar{\lambda}\right|}^n\left|(f,g)\right| \ge \left|(f,g)\right|, n\in \Z_+,
\]
and, for $\lambda\in\F$ with $|\lambda|<1$,
\[
\big|{\bar{\lambda}}^n(f,g)\big|={\left|\bar{\lambda}\right|}^n\left|(f,g)\right|  <\left|(f,g)\right|, n\in \Z_+.
\]

Thus, the adjoint operator $A^*$ has no eigenvalues.

The rest of the statement of part (1) immediately follows from the orthogonal sum decomposition
\begin{equation}\label{osd}
X=\ker(A^*-\bar{\lambda}I)\oplus \bar{R(A-\lambda I)},\ \lambda\in \F,
\end{equation}
the conjugation being superfluous when the space is real,
(see, e.g., \cite{Markin2026EOT}).
\item Suppose that the space $X$ is complex (i.e., $\F=\C$) and the operator $A$ is closed. 

Since
\begin{equation*}
\sigma_r(A)=\left\{\lambda\in \C \,\middle|\,A-\lambda I\ \text{is \textit{injective} and $\overline{R(A-\lambda I)}\neq X$} \right\}
\end{equation*}
(see, e.g., \cite{Markin2017,Markin2026EOT}), we infer from part (1) that
\[
\sigma_r(A)=\emps.
\]
\end{enumerate}
\end{proof}

We immediately obtain the following result.

\begin{prop}[Non-Hypercyclicity Test]\label{NHCT}
Any densely defined closed linear operator $A$ in a nonzero complex Hilbert space $X$ with a nonempty residual spectrum (i.e., $\sigma_r(A)\neq \emps$)
is not hypercyclic.
\end{prop}

We now prove the subsequent theorem.

\begin{thm}\label{Thm2} 
An arbitrary symmetric operator $A$ in a complex Hilbert space $X$ is not hypercyclic.
\end{thm}

\begin{proof}
Without loss of generality, we can regard the symmetric operator $A$ to be \textit{closed}. Indeed, since the adjoint operator $A^*$ is a closed extension of $A$:
\[
A\subseteq A^*,
\]
the operator $A$ is \textit{closable}, its closure $\bar{A}$ is also symmetric, and
\[
A\subseteq \bar{A}\subseteq A^*
\]
(see, e.g., \cite{Dun-SchI,Markin2026EOT}). The latter implies that each hypercyclic vector for $A$ is also hypercyclic for $\bar{A}$. Thus, the non-hypercyclicity of the closed symmetric operator $\bar{A}$ implies the non-hypercyclicity of $A$.

If both \textit{deficiency indices} of the operator $A$ are equal to zero, $A$ is \textit{self-adjoint} ($A=A^*$) (see, e.g., \cite{Akh-Glaz}), and hence, by Theorem \ref{Thm1}, is not hypercyclic.

If at least one of the \textit{deficiency indices} of the operator $A$ is nonzero, then
\[
{R(A-i I)}^\bot=\ker(A^*+iI)\ne \left\{0\right\}\quad \text{or}\quad 
{R(A+i I)}^\bot=\ker(A^*-iI)\ne \left\{0\right\}
\]
($i$ is the \textit{imaginary unit}), implying that 
\[
\bar{R(A-i I)}\ne X\quad \text{or}\quad \bar{R(A+i I)}\ne X,
\]
and hence,
\[
\sigma_r(A)\neq \emps
\]
(see, e.g., \cite{Akh-Glaz,Markin2026EOT}), and hence, by Proposition \ref{NHCT}, $A$ is not hypercyclic.
\end{proof}

\section{Some Examples}

\begin{exmps}\label{exmpscard}\
\begin{enumerate}[label={\arabic*.}]
\item In the complex Hilbert space $L_2({\mathbb R})$, the self-adjoint differential operator $A:=i\dfrac{d}{dx}$ with the domain
\[
D(A):=W_2^1({\mathbb R}):=\left\{f\in L_2(\R)\middle|f(\cdot)\in AC(\R),\ f'\in L_2(\R) \right\}
\] 
($AC(\cdot)$ is the set of \textit{absolutely continuous functions} on an interval) is non-hypercyclic by Theorem \ref{Thm1} (cf. \cite[Corollary $4.1$]{Markin2020(1)}).
\item In the complex Hilbert space $L_2(0,\infty)$, the symmetric differential operator $A:=i\dfrac{d}{dx}$ with the domain
\[
D(A):=\left\{f\in L_2(0,\infty)\middle|f(\cdot)\in AC[0,\infty),\ f'\in L_2(0,\infty),\ f(0)=0 \right\}
\] 
and deficiency indices $(0,1)$ is non-hypercyclic by Theorem \ref{Thm2}.
\item In the complex Hilbert space $L_2(0,2\pi)$, the symmetric differential operator $A:=i\dfrac{d}{dx}$ with the domain
\[
\quad\qquad
D(A):=\left\{f\in L_2(0,2\pi)\middle|f(\cdot)\in AC[0,2\pi],\ f'\in L_2(0,2\pi),\ f(0)=f(2\pi)=0 \right\}
\] 
and deficiency indices $(1,1)$ is non-hypercyclic by Theorem \ref{Thm2}.
\end{enumerate}

Cf. \cite[Sections $49$ and $80$]{Akh-Glaz}. 
\end{exmps}


 
\end{document}